\documentclass[11pt]{article}  

\usepackage{amsmath, amsthm, amsfonts, amssymb}
\usepackage[colorlinks=true,linkcolor=red!50!black,citecolor=green!50!black,urlcolor=blue!50!black]{hyperref}
\usepackage{indentfirst} 
\usepackage{marvosym}
\usepackage{enumitem}
\usepackage[a4paper, margin=2.4cm]{geometry}
\usepackage{xcolor}
\usepackage{microtype}

\newtheorem{lemma}{Lemma}[section]
\newtheorem{theorem}[lemma]{Theorem}
\newtheorem{proposition}[lemma]{Proposition}
\newtheorem{question}[lemma]{Question}

\renewenvironment{proof}[1][\proofname]{{\noindent\bf #1. }}{\qed}
\newtheorem{theoremletters}{Theorem}

\theoremstyle{definition}
\newtheorem{example}[lemma]{Example}

\newcommand{\op}{\operatorname}
\newcommand{\ce}[2]{\pmb{\op{C}}_{#1}(#2)}

\newcommand{\ze}[1]{\pmb{\op{Z}}(#1)}

\newcommand{\rad}[2]{\pmb{\op{O}}_{#1}(#2)}
\newcommand{\syl}[2]{\op{Syl}_{#1}\left(#2\right)}
\newcommand{\hall}[2]{\op{Hall}_{#1}\left(#2\right)}

\begin{document}

\title{\bf The number of composite class sizes in finite groups}

\author{\sc Vittorio Bagnara\thanks{Dipartimento di Matematica, Università di Salerno, Via Giovanni Paolo II, 132-84084 Fisciano, Italy. \newline \Letter: \texttt{vbagnara@unisa.it} \newline ORCID: 0009-0004-5670-9067}\; and Víctor Sotomayor\thanks{Departamento de Álgebra, Universidad de Granada, Avda. de Fuentenueva s/n, 18071 Granada, Spain. \newline
\Letter: \texttt{vsotomayor@ugr.es} \newline 
ORCID: 0000-0001-8649-5742 \newline \rule{6cm}{0.1mm}\newline
This research has been carried out during a visit of the first author to Departamento de Álgebra of Universidad de Granada, with the financial support of the Erasmus+ Traineeship Programme of the European Union. He wishes to acknowledge all the members of the department for the hospitality. \newline
}}

\date{\textit{\small In honor of Alan Camina on his 85th birthday}}

\maketitle

\begin{abstract}
\noindent We consider finite groups with at least three conjugacy class sizes that are composite numbers and we prove that, in that situation, the number of prime class sizes is bounded by the number of composite class sizes. The analogous result for (non-)prime-power class sizes is also addressed.

\medskip


\smallskip

\end{abstract}


\section{Introduction}

Throughout this paper, all groups considered are assumed to be finite. The study of arithmetical properties of the set $cs(G)$ of class sizes of a group $G$ is a well-established research theme, since it can reveal information regarding the algebraic structure of $G$. Recent research has focused on groups in which the presence of composite numbers in $cs(G)$ is scarce.  More concretely, Jiang \emph{et al.} analysed in \cite{JSZ} arithmetical and structural features of $G$ when $cs(G)$ contains exactly one composite integer, whilst Monetta and Sotomayor in \cite{MS} placed particular emphasis on the case where $cs(G)$ possesses exactly two composite numbers. Additionally, the structure of $G$ in the extreme situation where $cs(G)$ has no composite integer can be deduced from a result due to Chillag and Herzog in \cite{CH}.

Let $cs_c(G)$ be the subset of composite class sizes of $G$, and let $cs_p(G)$ be the subset of prime class sizes of $G$, so certainly $cs(G)=\{1\}\cup cs_p(G)\cup cs_c(G)$. By the main results in the aforementioned papers, the next conclusions particularly follow:

\begin{itemize}
\setlength{\itemsep}{0mm}
\item[(1)] If $cs_c(G)=\emptyset$, then $|cs_p(G)|\leq 2$, and this bound is sharp (see \cite[Corollary 2.3]{CH}).
\item[(2)] If $|cs_c(G)|=1$, then $|cs_p(G)|\leq 2$, and this bound is sharp (see \cite[Theorem A]{JSZ}).
\item[(3)] If $|cs_c(G)|=2$, then $|cs_p(G)|\leq 3$, and this bound is sharp (see \cite[Theorem A]{MS}).
\end{itemize}

Motivated by this development, in this note we consider a broader framework and investigate a bound for $|cs_p(G)|$ when the group $G$ has at least three composite class sizes, that is, when $|cs_c(G)|\geq 3$. Our main result shows that the previous arithmetical pattern changes:

\begin{theoremletters}
\label{theoA}
Let $G$ be a group and $n\geq 3$ be an integer. If $G$ has $n$ composite class sizes, then the number of prime class sizes of $G$ is at most $n$, and less than $n$ if $n\geq 5$.
\end{theoremletters}

As an immediate consequence, we obtain that if $G$ has $n\geq 3$ composite class sizes, then $|cs(G)|\leq 2n+1$, and $|cs(G)|<2n+1$ if $n\geq 5$. Hence we can control the number of class sizes by the amount of composite integers in $cs(G)$; this is certainly no longer true when one only considers the amount of prime numbers in $cs(G)$, since in virtue of the main result of \cite{CosseyHawkes} for any prime $p$ and any integer $k\geq 2$ there exists a $p$-group $G$ with $cs(G)=\{1,p,p^2,p^3,...,p^k\}$.

We remark that the bound given in Theorem~\ref{theoA} is attained when $n = 4$. In fact, it is enough to consider the direct product $G=A\times B$ of two groups such that $cs(A)=\{1,p,q\}$ and $cs(B)=\{1,r,s\}$, for pairwise different primes $p,q,r$ and $s$. It follows that $cs(G)=\{1,p,q,r,s,pr,ps,qr,qs\}$ so $|cs_p(G)|=4=|cs_c(G)|$. We point out that the structure of these direct factors is well-known: $A/\ze{A}$ and $B/\ze{B}$ are Frobenius groups of orders $pq$ and $rs$, respectively, and the inverse images of the kernels and complements are abelian (see Theorem~\ref{disconnected}). Regarding the case $n=3$ of Theorem~\ref{theoA}, we strongly believe that $|cs_p(G)|<|cs_c(G)|=3$, although we have been able to prove this only for some specific configurations (see Proposition~\ref{propC}).

It is worth mentioning that Theorem~\ref{theoA} fits into a broader setting; in fact, we prove it as a consequence of a more general result concerning the number of class sizes that are not prime powers:

\begin{theoremletters}
\label{theoB}
Let $G$ be a group and $n\geq 1$ be an integer. If $G$ has $n$ class sizes that are not prime powers, then the number of primes whose powers appear as class sizes of $G$ is at most $n + 1$; at most $n$ if $n \geq 3$, and less than $n$ if $n\geq 5$.
\end{theoremletters}

We recall that groups $G$ which do not have any non-prime-power class size (that is, the case $n=0$ in Theorem \ref{theoB}) were already characterised by Chillag and Herzog in \cite[Theorem 2]{CH}, and in particular they showed that the class sizes of $G$ are powers of at most two different primes. Observe that groups satisfying Theorem~\ref{theoB} with $n=1$ are necessarily soluble by \cite[Proposition 3.2]{CC}. Another interesting remark concerning Theorem~\ref{theoB} is that the bounds may also be attained with equality if $n\leq 4$ (see Examples \ref{ex1}, \ref{ex2} and \ref{ex3}).
 
To close with, we highlight that results about the set of class sizes of a group sometimes correspond to similar results about the set of irreducible character degrees. Theorem \ref{theoA} can be seen as another instance of such a correspondence, since its counterpart for the set of irreducible character degrees was addressed in \cite{A}, although for soluble groups only.


\section{Preliminaries}

Let $G$ be a group. In the following, we write $x^G$ for the conjugacy class of an element $x\in G$, and we write $\pi(x^G)$ for the set of prime divisors of $|x^G|=|G:\ce{G}{x}|$. We will frequently use that each element $g\in G$ can be decomposed as a product of pairwise commuting elements of prime power order, say $g_{q_1},...,g_{q_s}$, for some integer $s\geq 1$ and some set of distinct primes $\{q_1,...,q_s\}$; each element $g_{q_i}$ has order a power of $q_i$ and is called the \emph{$q_i$-part} of $g$, and such a factorisation is called the \emph{primary decomposition} of $g$. This decomposition is unique up to reordering of the factors, and can be seen by considering the decomposition of the cyclic group $\langle g \rangle$ into a direct product of cyclic groups of prime power order. The remaining notation and terminology used are standard in the framework of group theory.

Let us start by stating some elementary properties that will be frequently used.

\begin{lemma}
\label{basic}
Let $G$ be a group. 
\vspace*{-2mm}
\begin{itemize}
\setlength{\itemsep}{0mm}
\item[\emph{(a)}] If $x,y\in G$ satisfy $(|x^G|,|y^G|)=1$, then $(xy)^G=x^G y^G$, and $|(xy)^G|$ divides $|x^G|\cdot|y^G|$.
\item[\emph{(b)}] If $x,y\in G$ have coprime orders and they commute, then $\ce{G}{xy}=\ce{G}{x}\cap\ce{G}{y}$.
\item[\emph{(c)}] If $x,y\in G$ have coprime orders, coprime class sizes, and they commute, then $|x^G|\cdot |y^G|=|(xy)^G|$.
\item[\emph{(d)}] A prime $p$ does not divide any element of $cs(G)$ if and only if $G$ has an abelian Sylow $p$-subgroup $P$ and  $G=P\times \rad{p'}{G}$.
\end{itemize}
\end{lemma}

\begin{proof}
If $(|x^G|,|y^G|)=1$, then $G=\ce{G}{x}\ce{G}{y}$. Take $x^gy^h\in x^Gy^G$ for some $g,h\in G$. Set $hg^{-1}=uv$ for some $u\in\ce{G}{y}$ and $v\in\ce{G}{x}$. Hence $x^gy^h=x^{vg}y^{u^{-1}h}=(xy)^{vg}\in(xy)^G$, so $x^Gy^G\subseteq (xy)^G$; the other inclusion is trivial. In addition, since $\ce{G}{x}\cap \ce{G}{y}\leqslant\ce{G}{xy}\leqslant G$ it follows that $|G:\ce{G}{xy}|$ divides $$|G:\ce{G}{x}\cap \ce{G}{y}|= |G:\ce{G}{x}|\cdot |\ce{G}{x}:\ce{G}{x}\cap\ce{G}{y}| = |x^G|\cdot |y^G|,$$ as claimed in (a). If $x$ and $y$ are as in (b), then they are both powers of $xy$ and so $\ce{G}{xy}\leqslant\ce{G}{x}\cap\ce{G}{y}\leqslant\ce{G}{xy}$. Statement (c) is an immediate consequence of (a) and (b). Finally, for statement (d), observe that if each element $g\in G$ has class size not divisible by $p$, then $G$ can be written as a union of conjugates of $\ce{G}{P}$, for some Sylow $p$-subgroup $P$ of $G$, so $G=\ce{G}{P}$; the other implication is clear.
\end{proof}

The next results are helpful when dealing with class sizes that are coprime.

\begin{lemma}\emph{(\cite[Lemma 1]{CCcoprime})}
\label{lemma_CC}
Let $G$ be a group and $1<a<b_1<b_2$ be pairwise coprime elements of $cs(G)$. Then there exist some $i\in \{1,2\}$ and some $c\in cs(G)$ such that $c>b_i$ and $c$ divides $ab_i$.
\end{lemma}

\begin{theorem}\emph{(\cite[Theorem 4]{D})}
\label{disconnected}
Let $G$ be a group and $\pi$ be a set of primes. Suppose that each element of $cs(G)$ is either a $\pi$-number or a $\pi'$-number, and that both cases occur. Then, up to abelian direct factors (and up to interchanging $\pi$ and $\pi'$), $G=HL$ with $H\in \hall{\pi}{G}$, $L\in\hall{\pi'}{G}$, $L\unlhd G$, both $H$ and $L$ abelian, and $G/\ze{G}$ a Frobenius group. In particular, $cs(G)=\{1,|L|, |H/\ze{G}|\}$.
\end{theorem}

The following is a key ingredient to prove our main results.

\begin{lemma}\emph{(\cite[Lemma 4]{BK})}
\label{lemma_kazarin}
Let $G$ be a group and $t$ be a prime. Suppose that the $t$-elements $x,y \in G \setminus \ze{G}$ are such that $|x^G|$ and $|y^G|$ are powers of distinct primes, and $|(xy)^G|$ is also a power of a prime. Then $\langle x,y\rangle \leqslant \rad{t}{G}$, and $|(xy)^G|= \max\{|x^G|,|y^G|\}$ is a power of $t$.
\end{lemma}

The global information provided below on the $p$-structure of $G$ by the class sizes of its $p'$-elements will be essential at some points:

\begin{lemma}\emph{(\cite[Lemma 2.4]{JS})}
\label{lemmaJS}
Let $G$ be a group and $p$ be a prime. If every prime power order $p'$-element has class size $1$ or $m$, for some fixed positive integer $m$, then $m = p^aq^b$ for some prime $q \neq p$ and some integers $a,b \geq 0$. In particular, up to abelian direct factors, $G$ is a $\{p,q\}$-group.
\end{lemma}

Finally we state the following result regarding \emph{minimal} centralisers due to N. Itô. An elementary proof is included, for instance, in \cite[Lemma 2.6]{MS}.

\begin{lemma} \label{Ito}
    Let $G$ be a group. If $X$ is a centraliser that does not properly contain any other centraliser, and there exist a prime $r$ and an $r$-element $g \in G$ whose centraliser is $X$, then $X = R \times A$ where $R$
    is a Sylow $r$-subgroup of $X$ and $A$ is an abelian $r'$-group.
\end{lemma}

\section{Proofs of main results}

As mentioned in the introduction, we will prove Theorem \ref{theoA} as a consequence of some more general results. To formulate them, let us first introduce the following notation: let $cs_{pp}(G)$ be the set of distinct primes whose powers occur as class sizes in $G$, and denote by $cs_{npp}(G)$ the subset of $cs(G)$ consisting of those class sizes which are not prime powers.

\medskip

\begin{theorem}
\label{theoremB}
Let $G$ be a group.
\vspace*{-2mm}
\begin{itemize}
\setlength{\itemsep}{0mm}
\item[\emph{(a)}] If $|cs_{npp}(G)|\geq 3$, then $|cs_{pp}(G)| \leq |cs_{npp}(G)|$. 
\item[\emph{(b)}] If $|cs_{npp}(G)|\geq 5$, then $|cs_{pp}(G)| < |cs_{npp}(G)|$.
\end{itemize} 
\end{theorem}

\begin{proof}
(a) Set $n=|cs_{npp}(G)|\geq 3$. Suppose, arguing by contradiction, that there exist elements $x_1, x_2, \dots, x_{n+1} \in G$ such that for all $1\leq r \leq n+1$ we have $|x_r^G|=p_r^{a_r}$ for certain pairwise different primes $p_r$, where $p_r^{a_r}$ is the smallest non-trivial power of $p_r$ in $cs(G)$. Set $R=\{1,...,n+1\}$. In virtue of Lemma~\ref{basic} (b) and the primary decomposition of each $x_r$, we may assume for every $r\in R$ that the order of $x_r$ is a power of some prime $t_r$.

We split the proof into the following four possible cases, according to how the primes $t_1,...,t_{n+1}$ are distributed: they can be pairwise distinct, all equal, all equal except one, or there is a prime $t$ occurring exactly $s$ times with $1<s<n$. These cases cover all possibilities.

\smallskip

\noindent \textbf{\underline{Case I:}} the elements $x_1,...,x_{n+1}$ have pairwise coprime orders.

\smallskip

In other words, we are supposing that the primes $t_r$ are pairwise different, for all $r\in R$. Fix $i\in R$. Since $t_i$ is equal to at most one of $\{p_{1},\dots, p_{i-1}, p_{i+1}, \dots, p_{n+1}\}$, we get $x_ix_k^{g_k}=x_k^{g_k}x_i$ for some $g_k\in G$ and for every $k \in R \setminus \{i, i^*\}$, where, if necessary, $i^*$ is chosen such that $t_{i}=p_{i^*}\in \{p_{1},\dots, p_{i-1}, p_{i+1}, \dots, p_{n+1}\}$. Observe that $(x_ix_k^{g_k})^G=(x_ix_k)^G$ by Lemma~\ref{basic}(a), and $|(x_ix_k)^G|=p_{i}^{a_i}p_{k}^{a_k}$ by Lemma~\ref{basic}(c). Now fix arbitrary $j\in R\setminus\{i\}$ if $t_i\notin \{p_{1},\dots, p_{i-1}, p_{i+1}, \dots, p_{n+1}\}$, or $j=i^*$ if $t_{i}=p_{i^*}\in \{p_{1},\dots, p_{i-1}, p_{i+1}, \dots, p_{n+1}\}$. In the latter case we have thus built $n-1$ distinct class sizes of $G$ that are not prime powers, while in the former case we have built $n$. Arguing similarly, since $t_j$ is equal to at most one of the primes in $\{p_{1}, \dots, p_{j-1}, p_{j+1}, \dots, p_{n+1}\}$, we obtain conjugacy classes with sizes $p_{j}^{a_j}p_{s}^{a_s}$, for every $s \in R \setminus \{j, j^*\}$, where, if necessary, $j^*$ is chosen such that $t_{j}=p_{j^*}$. Observe that when $x_j=x_{i^*}$ all these class sizes $p_j^{a_j}p_s^{a_s}$ are distinct to the previous ones built, since $p_i^{a_i}p_k^{a_k}$ is never divisible by $p_{i^*}$ but $p_j^{a_j}p_s^{a_s}$ always is; but this is a contradiction because we have therefore created $(n-1)+(n-1)=2n-2>n$ different class sizes of $G$, none of them being a prime power. On the other hand, when $j\in R\setminus \{i\}$ is chosen arbitrarily, we obtain at least $n+(n-1)=2n-1>n$ different non-prime-power class sizes which is also not possible.

\smallskip

\noindent \textbf{\underline{Case II:}} all elements $x_1,...,x_{n+1}$ are $t$-elements, for a fixed prime $t$.

\smallskip

Since $\operatorname{gcd}(|x_{i}^G|, |x_{j}^G|)=1$, for every $i,j \in R$ with $i\neq j$, then by Lemma~\ref{basic} (a) one has that $|(x_{i}x_{j})^G|$ divides $|x_{i}^G|\cdot |x_{j}^G|=p_i^{a_i}p_j^{a_j}$, and clearly $1\neq |(x_{i}x_{j})^G|$. Recall that $p_i^{a_i}$ and $p_j^{a_j}$ are the smallest powers of $p_i$ and $p_j$, respectively, in $cs(G)$. Consequently $|(x_{i}x_{j})^G|$ is equal either to $p_i^{a_i}$, to $p_j^{a_j}$, or to a number divisible by both primes. Since there are $\frac{n(n+1)}{2}$ possible products $x_ix_j$, and by assumptions we have exactly $n$ non-prime-power numbers in $cs(G)$, then at least $\frac{n(n+1)}{2}-n = \frac{n(n-1)}{2}$ of these products $x_ix_j$ have prime power class size in $G$. In particular, if $n\geq 4$, then $\frac{n(n-1)}{2}>n$, so there are at least $n+1$ possible products $x_ix_j$ with prime power class size. Applying Lemma \ref{lemma_kazarin} all these class sizes are powers of the (fixed) prime $t$, and each one coincides with $\operatorname{max}\{|x_i^G|,|x_j^G|\}=\operatorname{max}\{p_i^{a_i}, p_j^{a_j}\}$, so $t\in\{p_i,p_j\}$ for each pair of indices $i\neq j$. But this cannot occur because there are at least $n+1$ different products $x_ix_j$ and $|R|=n+1$. Thus we may affirm $n=3$, and by the previous argument we deduce that within the set $$\{|(x_{1}x_{2})^G|, |(x_{1}x_{3})^G|,|(x_{1}x_{4})^G|,  |(x_{2}x_{3})^G|, |(x_{2}x_{4})^G|, |(x_{3}x_{4})^G| \}$$ there are three prime powers (in fact, $t$-powers, by Lemma~\ref{lemma_kazarin}), and the remaining three integers are not prime powers. In order to avoid contradicting Lemma~\ref{lemma_kazarin}, all the previous products of elements that possess prime power class size must share a common term, say $x_{1}$ without loss of generality. In other words, we have
    \begin{equation*}
        |(x_{1}x_{2})^G|= |(x_{1}x_{3})^G|=|(x_{1}x_{4})^G| =p_{1}^{a_{1}}=t^{a_1},
    \end{equation*}
and the remaining ones satisfy $|\pi((x_2x_3)^G)|=|\pi((x_2x_4)^G)|=|\pi((x_3x_4)^G)|=2$.

We claim that, if $|g^G|$ is a prime power for some $g \in G \setminus \ze{G}$, then the $q$-part $g_{q} \in \ze{G}$ for every prime divisor $q\neq p_1$ of the order of $g$. Otherwise $|g_{q}^G|\neq 1$, and since $|g_q^G|$ divides the prime power $|g^G|$, by Lemma~\ref{basic} (b) we get that both $|x_1^G|=p_1^{a_1}$ and $|g_q^G|$ divide $|(x_1g_q)^G|$. Since we already have $n=3$ non-prime-power numbers in $cs(G)$ and none of them is divisible by $p_1$, necessarily $\pi(g_q^G)=\{p_1\}$. Moreover, $q$ is necessarily different from one of $p_{2}, p_{3}, p_{4}$, so Lemma~\ref{basic}(a) and (c) yield $\pi((g_qx_l)^G)=\{p_1,p_l\}$ for some $l \in \{2,3,4\}$, which is a contradiction. 

As a consequence there is no $p_1'$-element in $G$ with prime power class size; equivalently the class size of every non-central $p_1'$-element $h\in G$ satisfies $\pi(h^G) \in \{\{p_2,p_3\},\{p_2,p_4\}, \{p_3,p_4\}\}$. Clearly there exists a non-central $p_1'$-element $h\in G$, as otherwise $G$ is a $p_1$-group up to abelian direct factors; therefore $x_1^{g_1}h=hx_1^{g_1}$ for some $g_1\in G$, and $|\pi((x_1h)^G)|=3$ by Lemma \ref{basic}(a) and (c), so we have more than $n$ class sizes that are not prime powers, a contradiction.

\smallskip

\noindent \textbf{\underline{Case III:}} all elements $x_1,...,x_n$ are $t$-elements, for a fixed prime $t$, but $x_{n+1}$ is not.

\smallskip

In other words, we are supposing that $t_1=\cdots = t_n=t$ but $t_{n+1}\neq t$. As $t_{n+1}$ is equal to at most one of $\{p_1, \dots, p_n\}$, we may assume that $t_{n+1} \notin \{p_1, \dots, p_{n-1}\}$ and so we get the non-prime-power class sizes $p_i^{a_i}p_{n+1}^{a_{n+1}}$ for all $i \in \{1, \dots, n-1\}$ by Lemma \ref{basic}. Now, since $|cs_{npp}(G)|=n$, if we consider $|(x_ux_v)^G|$ for any pair $u,v\in \{1,....,n\}$ with $u\neq v$, then there is at most one of these class sizes which is divisible by both primes $p_u$ and $p_v$, and the remaining are equal to either $p_u^{a_u}$ or $p_v^{a_v}$. If $n>3$, then we can conclude that at least $\frac{n(n-1)}{2}-1>n-1$ of these class sizes $|(x_ux_v)^G|$ are prime powers. In virtue of Lemma \ref{lemma_kazarin}, all these powers are necessarily $t$-powers, and each one coincides with $\operatorname{max}\{|x_u^G|,|x_v^G|\}=\operatorname{max}\{p_u^{a_u}, p_v^{a_v}\}$, which is impossible because there are at least $n$ different products $x_ux_v$ and only $n$ elements in $\{x_1,...,x_n\}$.

Consequently $|cs_{npp}(G)|=3$, and by the above argument we can similarly deduce that $cs_{npp}(G)$ contains $\{ p_1^{a_1}p_4^{a_4},p_2^{a_2}p_4^{a_4} \}$, so at most one of $\{|(x_1x_2)^G|, \allowbreak |(x_1x_3)^G|, |(x_2x_3)^G|\}$ also belongs to $cs_{npp}(G)$; in fact, there is exactly one of them lying in $cs_{npp}(G)$, since otherwise we get a contradiction again with Lemma~\ref{lemma_kazarin}. Observe also that this last non-prime-power class size is not divisible by $p_4$ in any case. Now Lemma~\ref{lemma_kazarin} applied to the pair of class sizes inside $\{|(x_1x_2)^G|, \allowbreak |(x_1x_3)^G|, |(x_2x_3)^G|\}$ that are prime powers provides $t \in \{p_1, p_2, p_3\}$. Thus $t \neq p_4$, so up to conjugation $x_3x_4=x_4x_3$ and we get $|(x_3x_4)^G|=p_3^{a_3}p_4^{a_4}$, a contradiction.

\smallskip

\noindent \textbf{\underline{Case IV:}} all elements $x_1,...,x_s$ are $t$-elements, for a fixed prime $t$ and for some $1<s<n$, but $x_{s+1},...,x_{n+1}$ are not.

\smallskip

Set $I=\{1,...,s\}$ and $J=\{s+1,...,n+1\}$. We claim $t\in\{p_{s+1},...,p_{n+1}\}$. Arguing by contradiction, suppose the opposite, so up to conjugation we may affirm that $x_{i}x_{j}=x_{j}x_{i}$, for each $(i,j)\in I\times J$. By Lemma~\ref{basic} we get $|(x_{i}x_{j})^G|=p_{i}^{a_i}p_{j}^{a_j}$, so we have obtained $s(n-s+1)$ distinct non-prime-power class sizes. However, this number is strictly greater than $n$, since $s(n-s+1)>n$ is equivalent to $s^2-s(n+1)+n<0$, which holds if and only if $1<s<n$. This contradiction leads to $t\in\{p_{s+1},...,p_{n+1}\}$, so we may assume $t=p_{s+1}$. Hence, we can obtain (at least) $|I|(|J|-1)=s(n-s)$ different class sizes of $G$, with the form $p_i^{a_i}p_j^{a_j}$ for every $(i,j)\in I\times (J\setminus\{s+1\})$. Moreover, as $t_{s+1}$ is equal to at most one of $\{p_1,...,p_s\}$, then we get additional non-prime-power class sizes, which are $p_i^{a_i}p_{s+1}^{a_{s+1}}$, except possibly $p_{s+1}^{a_{s+1}}p_{i^*}^{a_i^*}$ for some $i^*\in I$ such that $t_{s+1}=p_{i^*}$. This gives us a total amount of, at least, $s(n-s)+(s-1)$ non-prime-power class sizes of $G$. Furthermore, the fact that $|(x_ux_v)^G|$ divides $p_u^{a_u}p_v^{a_v}$ for every pairwise distinct indices $u,v\in I$ forces that each of these numbers is equal to either $p_u^{a_u}$ or $p_v^{a_v}$; otherwise we have obtained an extra non-prime-power class size, which adds up a total of, at least, $s(n-s)+(s-1)+1$, which is larger than $n$ due to our assumption $1<s<n$, so this is not possible. Consequently, there is some pair $u\neq v$ of indices in $I$ such that $|(x_ux_v)^G|$ is a prime power, and Lemma \ref{lemma_kazarin} leads to the final contradiction that $\op{max}\{p_u^{a_u},p_v^{a_v}\}$ is a power of $t\in\{p_{s+1},...,p_{n+1}\}$.

\bigskip

\noindent (b) Using statement (a), and arguing by contradiction, suppose that $|cs_{pp}(G)|=|cs_{npp}(G)|=n\geq 5$, so $cs(G)$ contains non-trivial prime powers $\{p_1^{a_1},p_2^{a_2},...,p_n^{a_n}\}$ for some pairwise distinct primes $p_i$ and some positive integers $a_i$. We may certainly assume $p_i^{a_i}<p_j^{a_j}$ for all $i<j$. By Lemma~\ref{lemma_CC}, for every pair of indices $1<i<j$, there exists some $k\in\{i,j\}$ and some $c_k\in cs(G)$ such that $c_k>p_k^{a_k}$ and $c_k$ divides $p_1^{a_1}p_k^{a_k}$. In particular $\pi(c_k)=\{p_1p_k\}$, so $c_k\in cs_{npp}(G)$. Consequently, it is easy to realise that there are at least $n-2$ pairwise different indices $k\in\{2,...,n\}$ such that $c_k\in cs_{npp}(G)$. We can argue similarly with all possible indices $2<i<j$, and by Lemma~\ref{lemma_CC} we get non-prime-power class sizes $d_l$ with $\pi(d_l)=\{p_2,p_l\}$, for at least $n-3$ different indices $l\in\{3,...,n\}$; note that these last non-prime-power class sizes $d_l$ of $G$ are different from the previous $c_k$ that we built, since $p_1\in\pi(c_k)$ for each $k\in\{2,...,n\}$ but $p_1\notin \pi(d_l)$ for every $l\in\{3,...,n\}$. Hence $|cs_{npp}(G)|\geq (n-2)+(n-3)=2n-5$. This is certainly impossible if $n>5$, so we may suppose $n=5$. But as $p_3<p_4<p_5$ we can apply the same previous lemma and we get a new non-prime-power class size, which cannot occur since it is clearly different from all the elements that already lie in $cs_{npp}(G)$. 
\end{proof}

\bigskip

\begin{example}
\label{ex1}
It is not difficult to find groups for which the bound presented in Theorem~\ref{theoremB}(a) is attained with equality. For $n=4$ we have already shown an example in the introduction. For $n=3$, if we take $A$ and $B$ groups with $cs(A)=\{1,p,q\}$ and $cs(B)=\{1,p,r\}$, respectively, for pairwise different primes $p,q,r$, then $G=A\times B$ has $cs(G)=\{1,p,q,r,p^2,pq,pr,qr\}$, so $cs_{pp}(G)=\{p,q,r\}$ and $cs_{npp}(G)=\{pq,pr,qr\}$. \hfill $\blacksquare$
\end{example}

In view of the previous theorem, it is also natural to analyse the behaviour of $|cs_{pp}(G)|$ when $|cs_{npp}(G)|<3$. The extreme case when $cs_{npp}(G)=\emptyset$ was characterised by Chillag and Herzog in \cite[Theorem 2]{CH}, and in particular they showed $|cs_{pp}(G)|\leq 2$. 

The next case to study is when $|cs_{npp}(G)|=1$, which has not been analysed yet, up to our knowledge. In that situation it holds the so-called \emph{one-prime power hypothesis}, that is, the greatest common divisor between each pair of class sizes of $G$ is either $1$ or a prime power; thus \cite[Proposition 3.2]{CC} ensures that $G$ is soluble, although no information is known about $|cs_{pp}(G)|$. We demonstrate below that this number can indeed be bounded by $2$, and we provide some features of the group when the bound is attained.

\medskip

\begin{proposition}
\label{propB1}
Let $G$ be a group with $cs_{npp}(G)=\{m\}$. Then $|cs_{pp}(G)| \leq 2$. Moreover, if $cs_{pp}(G)=\{p,q\}$ for some primes $p\neq q$, then $\pi(m)=\{p,q\}$, and so $G$ is a $\{p,q\}$-group up to abelian direct factors.
\end{proposition}

\begin{proof}
For the first assertion, and arguing by contradiction, let us suppose that $\{p, q, r\} \subseteq cs_{pp}(G)$ for pairwise distinct primes $p,q,r$. Hence there exist $x,y,z\in G$ such that $|x^G|=p^a$, $|y^G|=q^b$ and $|z^G|=r^c$, these being the smallest powers of $p$, $q$ and $r$ appearing in $cs(G)$, respectively. By the primary decomposition, we can take $x,y,z$ as elements with prime power orders. Without loss of generality, assume $p^a<q^b<r^c$. Hence by Lemma~\ref{lemma_CC} there exists a class size $|g^G|$ such that, either $|g^G|>q^b$ and $|g^G|$ divides $p^aq^b$, or $|g^G|>r^c$ and $|g^G|$ divides $p^ar^c$. Certainly, in the former case $\pi(g^G)=\{p,q\}$, whilst in the latter case $\pi(g^G)=\{p,r\}$; and in both cases $cs_{npp}(G)=\{|g^G|\}$. 

We claim that $\pi(g^G)=\{p,r\}$ cannot occur. Otherwise, since $|cs_{npp}(G)|=1$ by assumptions, then there is no class size in $G$ divisible by $p$ and $q$, and \cite[Lemma 2]{CCnil} forces that $x$ is a $q$-element. Similarly, since there is no class size in $G$ divisible by $q$ and $r$, then $y$ is an $r$-element by the same lemma. Therefore $yx=xy$ up to conjugation, and by Lemma~\ref{basic} we obtain the contradiction $|(xy)^G|=p^aq^b\in cs_{npp}(G)$. 

Hence necessarily $\pi(g^G)=\{p,q\}$, and so there is no class size in $G$ divisible by $pr$ or $qr$. In this case \cite[Lemma 2]{CCnil} yields that $x$ and $y$ are $r$-elements. If $z$ is not an $r$-element, then up to conjugation $z$ commutes with either $x$ or $y$, and so we would get in $cs_{npp}(G)$ either $|(xz)^G|=p^ar^c$ or $|(yz)^G|=q^br^c$, a contradiction. Thus $z$ is also an $r$-element.

We claim that, if $|h^G|$ is prime power for some $h \in G \setminus \ze{G}$, then the $l$-part $h_{l} \in \ze{G}$ for every prime divisor $l\neq r$ of the order of $h$. Otherwise $|h_{l}^G|$ divides $|h^G|$, so it is prime power, and by Lemma~\ref{basic} we have that $|z^G|=r^c$ and $|h_l^G|$ divide $|(zh_l)^G|$. Since the unique non-prime-power class size of $G$ is $g^G$ and $\pi(g^G)=\{p,q\}$, then $|h_l^G|$ is a non-trivial $r$-power. Moreover, $l$ is different from one of $\{p,q\}$, so Lemma~\ref{basic} yields either $\pi((xh_l)^G)=\{p,r\}$ or $\pi((yh_l)^G)=\{q,r\}$, which is impossible.

As a consequence we may affirm that there is no prime power order $r'$-element in $G$ with prime power class size. Clearly there exists a non-central prime power order $r'$-element $h\in G$, as otherwise $G$ is an $r$-group up to abelian direct factors. It follows that $|h^G|=m=|g^G|$, and $zh=hz$ up to conjugation, so $|(zh)^G|$ is another non-prime-power class size with $\pi((zh)^G)=\{p,q,r\}$. This final contradiction proves that $|cs_{pp}(G)| \leq 2$, as desired.

Next we demonstrate that if $cs_{npp}(G)=\{m\}$ and $|cs_{pp}(G)|=2$, then $\pi(m)= cs_{pp}(G)$. Set $cs_{pp}(G)=\{p,q\}$, and let $p^a$ and $q^b$ be the smallest powers of $p$ and $q$, respectively, occurring in $cs(G)$. Suppose by contradiction that $m$ is not a $\{p,q\}$-number or, equivalently, that there exists a prime $r \in \pi(m) \setminus \{p,q\}$. By assumption, there exists $x,y \in G$ such that $|x^G|=p^a$ and $|y^G|=q^b$. In virtue of Lemma~\ref{basic} and the primary decomposition, we may assume that $x$ is a $t_1$-element and $y$ is a $t_2$-element, for some primes $t_1,t_2$.

Let $R\in\syl{r}{G}$ such that, up to conjugation, $R\leqslant\ce{G}{x}\cap\ce{G}{y}$. By Lemma~\ref{basic}(d) it holds $R\nleqslant \ze{G}$. We claim that there must exist some element $g\in R\setminus \ze{G}$ with class size $m$. Otherwise our assumptions and \cite[Theorem A(b)]{CCimplications} yield that all class sizes of elements $g\in R\setminus \ze{G}$ are powers of a fixed prime, and such prime must be either $p$ or $q$. Without loss of generality, let us assume that every element $g\in R\setminus \ze{G}$ has $p$-power class size. Since $yg=gy$, and we are supposing $\pi(m)\neq \{p,q\}$, then necessarily $t_2=r$. But then $y$ is an $r$-element with $q$-power class size, which is impossible.

In virtue of the previous paragraph, we can pick some $g\in G\setminus \ze{G}$ with class size $m$. Then $\ce{G}{g}$ does not properly contain any other centraliser, and as $g$ is an $r$-element, by Lemma~\ref{Ito} we get $\ce{G}{g}=R_0\times A$ where $R_0\in\syl{r}{\ce{G}{g}}$ and $A$ is an abelian group. Recall that $x,y\in\ce{G}{g}$, which is nilpotent, so if both $x$ and $y$ have coprime orders, then $xy=yx$ and so $\pi(m)=\pi((xy)^G)=\{p,q\}$, a contradiction. Consequently $x$ and $y$ are both $t$-elements, for some prime $t$. By Lemma~\ref{basic} one has that $|(xy)^G|$ divides $|x^G|\cdot|y^G|=p^aq^b$. Since $cs_{npp}(G)=\{m\}$ and we are supposing that $m$ is not a $\{p,q\}$-number, then $|(xy)^G|$ should be a prime power. Therefore Lemma~\ref{lemma_kazarin} forces that $|(xy)^G|=\operatorname{max}\{|x^G|,|y^G|\}$ is a power of $t$, so $t\in\{p,q\}$.

Next we claim that every prime power order $t'$-element $h\in G\setminus \ze{G}$ satisfies $|h^G|=m$. Otherwise $|h^G|$ is either a $p$-power, or a $q$-power, and in both cases $h\in\ce{G}{g}$ up to conjugation. So $h$ commutes with both $x$ and $y$, and its order is coprime with $t$, so in both cases we can construct a non-prime-power class size whose prime divisors are only $p$ and $q$, a contradiction. 

As a consequence, every prime power order $t'$-element of $G$ has either $1$ or $m$ conjugates. By Lemma~\ref{lemmaJS}, we get that $|G/\ze{G}|$ is divisible by at most two different primes, which is the final contradiction. Finally, if $cs_{pp}(G)=\{p,q\}$ and $cs_{npp}(G)=\{m\}$ with $\pi(m)=\{p,q\}$, then by Lemma~\ref{basic}(d) one can easily see that $G$ is a $\{p,q\}$-group up to abelian direct factors.
\end{proof}

\bigskip

\begin{example}
\label{ex2}
To find a group that verifies the hypotheses of Proposition~\ref{propB1} take, for instance, the following one. Consider a fixed-point-free action of a cyclic group $A=\langle a\rangle$ of order $5$ on an elementary abelian group $B=\langle b_1,b_2,b_3,b_4\rangle$ of order $16$. Set $H=B\rtimes A$. Consider the following action of a cyclic group $C=\langle c \rangle$ of order $2$ on $H$: $\: a^c=a^{-1}, b_1^c=b_1b_2, b_2^c=b_2,  b_3^c=b_1b_2b_4, b_4^c=b_1b_3$. Then $G=H\rtimes C$ has $cs(G)=\{1,5,20,32\}$, so $cs_{pp}(G)=\{2,5\}$ and $cs_{npp}(G)=\{20\}$, and effectively $G$ is a $\{2,5\}$-group. 

Another example for Proposition~\ref{propB1} is provided by the next group $G$. Take $A=\langle a_1\rangle \times \langle a_2\rangle$ a direct product of a cyclic group of order $9$ and a cyclic group of order $3$, and consider the action of a cyclic group $B=\langle b\rangle$ of order $3$ such that $a_1^b=a_1a_2$ and $a_2^b=a_2$. Set $H=A\rtimes B$. Consider the action of a cyclic group $C=\langle c \rangle$ of order $2$ on $H$ such that $b^c=b$, $a_1^c=a_1^8$ and $a_2^c=a_2^{-1}$. Then $G=H\rtimes C$ has $cs(G)=\{1,2,3,6,27\}$, so $cs_{pp}(G)=\{2,3\}$ and $cs_{npp}(G)=\{6\}$, and clearly $G$ is a $\{2,3\}$-group.

The dihedral group of order $30$ shows that if $cs_{pp}(G)=\{p\}$ in Proposition~\ref{propB1}, then $p$ may not divide $m$; in fact, for such dihedral group $G$, we have $cs(G)=\{1,2,15\}$, so $cs_{pp}(G)=\{2\}$ and $cs_{npp}(G)=\{15\}$. \hfill $\blacksquare$
\end{example}

Next we move a step further and, using similar techniques, we give a bound for $|cs_{pp}(G)|$ in the remaining case when $|cs_{npp}(G)|=2$.

\medskip

\begin{proposition}
\label{propB2}
Let $G$ be a group with $cs_{npp}(G)=\{m,n\}$. Then $|cs_{pp}(G)| \leq 3$. Moreover, if $cs_{pp}(G)=\{p,q,r\}$ for some pairwise different primes $p, q, r$, then $\pi(m)\cup\pi(n)\subseteq\{p,q,r\}$, and so $G$ is a $\{p,q,r\}$-group up to abelian direct factors.
\end{proposition}

\begin{proof}
We first demonstrate that $cs_{pp}(G)$ cannot have four elements whenever $|cs_{npp}(G)|=2$. For that aim, arguing by contradiction, let us suppose that there exist prime powers $p_1^{a_1}<p_2^{a_2}<p_3^{a_3}<p_4^{a_4}$ in $cs(G)$, for some pairwise different primes $p_i$ and some positive integers $a_i$. Hence we can now apply Lemma~\ref{lemma_CC} to $p_1^{a_1}<p_2^{a_2}<p_3^{a_3}$ which provides a class size $m\in cs(G)$ that divides $p_1^{a_1}p_i^{a_i}$, for some $i\in \{2,3\}$, with $m>p_i^{a_i}$. Thus $\pi(m)=\{p_1,p_i\}$ necessarily. The same argument applied to $p_2^{a_2}<p_3^{a_3}<p_4^{a_4}$ yields a class size $n$ such that $\pi(n)=\{p_2,p_j\}$ for some $j\in \{3,4\}$. In particular, due to our assumptions, $cs_{npp}(G)=\{m,n\}$. If $\gcd(m,n)=1$ (\emph{i.e.} $i=3$ and $j=4$), then we can apply Theorem~\ref{disconnected} since all class sizes of $G$ are either $\{p_1,p_3\}$-numbers or $\{p_1,p_3\}'$-numbers, and we get the contradiction $|cs(G)|=3$. On the other hand, if $\gcd(m,n)>1$ (\emph{i.e.} they are both divisible by either $p_2$ or $p_3$), then we also get a contradiction due to Theorem~\ref{disconnected} because in this case all class sizes are either $\{p_1,p_2,p_j\}$-numbers or $\{p_1,p_2,p_j\}'$-numbers.

Let us now prove that if $cs_{pp}(G)=\{p,q,r\}$, then $\pi(m)\cup\pi(n)\subseteq\{p,q,r\}$. Let $p^a,q^b,r^c$ be the smallest powers of $p,q,r$, respectively, occurring in $cs(G)$. We may then suppose that $|x^G|= p^a, |y^G|= q^b, |z^G|= r^c$ where $x$ is a $t_1$-element, $y$ is a $t_2$-element and $z$ is a $t_3$-element, for certain primes $t_1, t_2,t_3$. We first claim that the thesis about $m$ and $n$ holds whenever the primes $t_i$ are pairwise distinct. In fact, since $t_1$ is different from at least one of $\{q,r\}$, say $t_1 \neq q$, then up to conjugation $xy=yx$, and by Lemma~\ref{basic} we get $|(xy)^G|= p^aq^b= m$. Arguing similarly, since $t_3$ is different from at least one of $\{p,q\}$, then we obtain a conjugacy class with size either $p^ar^c = n$ or $q^br^c = n$, as desired. 

Hence we may suppose, without loss of generality, that $t_1 = t_2 = t$ for some prime $t$. Now note that if $t \neq t_3$, since $t_3$ is different from one of $\{p,q\}$, then we obtain that either $|(xz)^G|= p^ar^c$ or $|(yz)^G|= q^br^c$. In particular, if $t_3$ is different from both $p$ and $q$, then we are done. Therefore, we have the following two cases to discuss: either $x,y,z$ are $t$-elements for some fixed prime $t$, or $x$ and $y$ are $t$-elements and $z$ is, say, a $q$-element with $t\neq q$. In the latter case, up to conjugation, $x$ and $z$ commute so it follows $|(xz)^G|= p^ar^c = m$. If moreover $t \neq r$, then we can obtain the class size $|(yz)^G| = q^br^c = n$, and we are done. So we may suppose that both $x$ and $y$ are $r$-elements. We can see that $|(xy)^G|$ divides $|x^G|\cdot |y^G|= p^aq^b$ by Lemma~\ref{basic}, so either $|(xy)^G|=p^a$, $|(xy)^G|=q^b$, or $|(xy)^G|=n$ and $\pi(n)=\{p,q\}$. However, the first two cases cannot occur, otherwise $|(xy)^G|= \max \{p^a,q^b\}$ would be a power of $r$ by Lemma~\ref{lemma_kazarin}, which is a contradiction. Thus $\pi(n)=\{p,q\}$ as claimed.

Thus we may assume hereafter that $x,y,z$ are $t$-elements for some prime $t$. Observe that, by Lemma~\ref{basic}, each element within $\{|(xy)^G|,|(yz)^G|,|(xz)^G|\}$ is either a prime power or divisible by two primes in $\{p,q,r\}$. Since $cs_{npp}(G)=\{m,n\}$ by assumptions, then certainly at least one of them should be a prime power. Further, they cannot all be prime powers, because in that case Lemma~\ref{lemma_kazarin} yields that $|(xy)^G|= \max \{p^a,q^b\}, |(yz)^G|= \max \{q^b,r^c\}$ and $|(xz)^G|= \max \{p^a,r^c\}$ are powers of $t$, which is impossible. Additionally, if only one of $\{|(xy)^G|,|(yz)^G|,|(xz)^G|\}$ is a prime power, then we are done. Consequently, it is enough to focus on the situation where two of them are prime powers and the other is not. By symmetry, let us suppose that $|(xy)^G|= m$, so $\pi(m)=\{p,q\}$, and that both $|(yz)^G|$ and $|(xz)^G|$ are prime powers. Then, Lemma~\ref{lemma_kazarin} forces that $|(yz)^G|= \max \{q^b,r^c\} = \max \{p^a,r^c\}= |(xz)^G|$ is a power of $t$. Hence $t=r$ and in particular $x,y,z$ are $r$-elements.

We claim that if $|h^G|$ is a prime power for some $h \in G \setminus \ze{G}$, then we may assume that the $l$-part $h_l \in \ze{G}$ for every $l \neq r$. Otherwise $|h_l^G|$ divides $|h^G|$, which is a prime power. If $|h_l^G|$ is a $p$-power or a $q$-power, then $hz= zh$ up to conjugation, so we can produce a conjugacy class size, which should indeed be $n$, divisible by either $pr$ or $qr$, as wanted. On the other hand, if $|h_l^G|$ is an $r$-power, since $l$ is different from either $p$ or $q$, it is easy to see that then we can also obtain a class size divisible by either $pr$ or $qr$, and this class size is in fact $n$, as desired.

So we may assume that there is no $r'$-element of prime power order in $G$ with prime power class size. Recall we have already showed $\pi(m)=\{p,q\}$. Let us suppose, arguing by contradiction, that there exists a prime $s\in\pi(n)\setminus\{p,q,r\}$. Therefore we can pick a non-central element $g\in S\in\syl{s}{G}$. Since $g$ is a prime power order $r'$-element of $G$, then its class size is either $m$ or $n$. If $|g^G|=m$, as $\pi(m)=\{p,q\}$, then up to conjugation $gz=zg$ and thus $n=|(gz)^G|$ with $\pi((gz)^G)=\{p,q,r\}$ by Lemma~\ref{basic}, a contradiction. It follows that $|g^G|=n$ and Lemma~\ref{Ito} leads to $\ce{G}{g}=S_0\times A$ with $S_0\in\syl{s}{\ce{G}{g}}$ and $A$ an abelian group. If $m$ does not occur as a class size of any prime power order $r'$-element of $G$, then Lemma~\ref{lemmaJS} leads to $|\pi(G/\ze{G})|\leq 2$ which is impossible. So there must exist a prime power order $r'$-element $h\in G$ with $|h^G|=m$. As $s\notin\pi(m)$ then up to conjugation we may assume $h\in\ce{G}{g}$, which is nilpotent. Moreover, as $|z^G|=r^c$ is not divisible by $s$, up to conjugation $z$ lies in $A$. Consequently $zh=hz$ and so $|(zh)^G|=n$ with $\pi(n)=\{p,q,r\}$. This final contradiction ends the proof, since the last claim regarding the prime divisors of $G/\ze{G}$ immediately follows from Lemma~\ref{basic} (d).
\end{proof}

\bigskip

\begin{example}
\label{ex3}
To illustrate an example regarding Proposition~\ref{propB2}, it is enough to consider a direct product $G=A\times B$ where $cs(A)=\{1,p,q\}$ and $cs(B)=\{1,r\}$, for pairwise different primes $p,q,r$, since in that situation $cs(G)=\{1,p,q,r,pr,qr\}$, so $cs_{npp}(G)=\{pr,qr\}$ and $cs_{pp}(G)=\{p,q,r\}$.

Moreover, if $|cs_{pp}(G)|<3$ in Proposition~\ref{propB2}, then the unique non-prime-power class sizes $m$ and $n$ may be divisible by primes that do not appear in $cs_{pp}(G)$. For instance, take a cyclic group $A=\langle a\rangle \times \langle b\rangle$ of order $15$, where $o(a)=5$ and $o(b)=3$, and let $C=\langle c\rangle$ be a cyclic group of order $4$. Set $G=A\rtimes C$, where $a^c=a^3$ and $b^c=b^2$. In this case $cs(G)=\{1,2,4,5,10,15\}$, so $cs_{npp}(G)=\{10,15\}$ but $15$ is divisible by $3\notin cs_{pp}(G)=\{2,5\}$. This may also occur if $|cs_{pp}(G)|=1$, as the next group shows. Let $G=(\langle a,b\rangle \times \langle c \rangle)\rtimes \langle d\rangle$ where $\langle a,b\rangle$ is a dihedral group of order $8$, $o(c)=5$, $o(d)=2$, $a^d=a^3$, $b^d=ba$ and $c^d=c^4$. Hence $cs(G)=\{1,2,4,10,20\}$, so $cs_{npp}(G)=\{10,20\}$ and both numbers are divisible by $5\notin cs_{pp}(G)=\{2\}$. \hfill $\blacksquare$
\end{example}

It is straightforward to see that Theorem~\ref{theoB} can be derived from Theorem~\ref{theoremB}, and Propositions~\ref{propB1} and \ref{propB2}. We are now ready to prove Theorem~\ref{theoA}.

\bigskip

\begin{proof}[Proof of Theorem \ref{theoA}]
Let us suppose that $|cs_c(G)|=n \geq 3$. Clearly $|cs_p(G)|\leq |cs_{pp}(G)|$, and $|cs_{npp}(G)|\leq |cs_c(G)|$ so $0 \leq |cs_{npp}(G)| \leq n$. If $|cs_{npp}(G)|=0$, then all class sizes of $G$ are prime powers, so applying \cite[Corollary 2.1]{CH} we get $|cs_{pp}(G)|\leq 2$; hence $|cs_p(G)|\leq |cs_{pp}(G)|\leq 2 < n = |cs_c(G)|$, as wanted. Suppose now $|cs_{npp}(G)|=1$. Then by Proposition \ref{propB1} we get $|cs_{pp}(G)|\leq 2$, and it follows $|cs_p(G)| \leq 2 < |cs_c(G)|$. Using similar arguments and Proposition \ref{propB2}, one can see that if $|cs_{npp}(G)|=2$, then $|cs_p(G)| \leq 3 \leq |cs_c(G)|$. Finally, when $|cs_{npp}(G)| \geq 3$, we can apply Theorem~\ref{theoremB} to deduce $|cs_p(G)|\leq |cs_{pp}(G)|\leq |cs_{npp}(G)|\leq |cs_c(G)|$, and in particular if $|cs_{npp}(G)|\geq 5$ then $|cs_p(G)|\leq |cs_{pp}(G)|< |cs_{npp}(G)|\leq |cs_c(G)|$, as desired.
\end{proof}

\bigskip

In Example~\ref{ex1} we mentioned that there are groups $G$ with $|cs_{pp}(G)|=3=|cs_{npp}(G)|$, so the bound presented in Theorem~\ref{theoremB}(a) may hold with equality when $n=3$. However, we have not been able to determine whether this also occurs in Theorem~\ref{theoA} when $n=3$:

\begin{question}
\label{quest}
Does there exist a group $G$ with $|cs_p(G)|=3=|cs_c(G)|$?
\end{question}

We strongly believe that the previous question has a negative answer. We finish this note by providing some evidence for this problem.

\bigskip

\begin{proposition}
\label{propC}
There is no finite group $G$ with $cs(G)=\{1,p,q,r,pq,pr,qr\}$, where $p,q,r$ are pairwise different primes.
\end{proposition}

\begin{proof}
We utilise the next notation: if $\pi_1,\pi_2$ are non-empty sets of positive integers, then we set $\pi_1\times \pi_2 = \{n_1\cdot n_2 \; | \; n_1\in \pi_1,\: n_2\in \pi_2\}$. Observe now that, if there exists such a group $G$, then $\{p,q,r\}\subseteq cs(G) \subseteq \{1,p\}\times\{1,q\}\times\{1,r\}$. Thus \cite[Theorem 1.2]{CT} implies that there are two possibilities (up to
abelian direct factors):
\begin{itemize}
\setlength{\itemsep}{-1mm}
\item $G=P\times H$ where $cs(P)=\{1,p\}$ and $cs(H)=\{1,q,r\}$, or
\item $G = P\times Q\times R$ with $cs(P)=\{1,p\}$, $cs(Q)=\{1,	q\}$ and $cs(R)=\{1,r\}$.
\end{itemize}
Since we do not get in any case the original set $cs(G)$, then such $G$ cannot exist.
\end{proof}

\bigskip

In particular, if $|cs_p(G)|=3=|cs_c(G)|$ for some finite group $G$, and $cs_p(G)=\{p,q,r\}$, then $cs_c(G)\nsubseteq \{pq,pr,qr,pqr\}$.


\bigskip
\bigskip

\noindent \textbf{Declarations.} The authors declare no conflict of interest.


\end{document}